\documentclass[a4paper,11pt,reqno]{amsart}

\usepackage{amsmath,amssymb}
\usepackage{enumerate}
\usepackage{ifthen}
\usepackage{enumitem}
\usepackage{graphicx}
\usepackage{tikz}
\usepackage{tikz-qtree}
\usepackage[hypcap]{caption}
\usetikzlibrary{arrows}
\tikzstyle{block}=[draw opacity=0.7,line width=1.4cm]
\nonstopmode\numberwithin{equation}{section}
\usepackage{url}
\newtheorem{conj}{Conjecture}

\allowdisplaybreaks

 \theoremstyle{plain}

\newtheorem{prop}{Proposition}

\begin{document}
\title{ A note on an Inverse problem in Algebra}
\maketitle
\begin{center}
\center{$ \text{Gaurav Mittal}$} 
{\footnotesize
 \center{ Department of Mathematics, Indian Institute of Technology Roorkee, Roorkee, India
 
 \email {gmittal@ma.iitr.ac.in} }
}

\smallskip
\begin{abstract}   
In this paper, we discuss the inverse problem of determining a semisimple group algebra from the knowledge of rings of the type $\oplus_{t=1}^jM_{n_t}(\mathbb{F}_t)$, where $j$ is an arbitrary integer and $\mathbb{F}_t$ is  finite field for each $t$, and show that it is ill-posed. After then, we   define the concept of completeness of the rings of the type $\oplus_{t=1}^jM_{n_t}(\mathbb{F}_t)$  to pose a well-posed inverse problem  and propose a  conjecture  in this direction.
\end{abstract}
\end{center}
\subjclass \textbf{Mathematics Subject Classification (2010)}: {16U60, 20C05.}

\keywords{\textbf{Keywords:}  Group algebra, Wedderburn decomposition
\medskip
\section{Main result}
For semisimple group algebras, determination of Wedderburn decomposition $[1]$ is a classical problem which is further incorporated to obtain the unit groups.  Mathematically,  if $\mathbb{F}_qG$ is the finite semisimple group algebra over a field with $q=p^k$ elements for some prime $p$, then $$\mathbb{F}_qG\cong \oplus_{t=1}^jM_{n_t}(\mathbb{F}_t), \quad \text{and}\quad U(\mathbb{F}_qG)\cong \oplus_{t=1}^jGL_{n_t}(\mathbb{F}_t) $$ where for each $t$, $\mathbb{F}_t$ is a finite extension of $\mathbb{F}_q$ and $n_t, j \in \mathbb{Z}$. Units of group rings play an important role in various fields of mathematics including coding theory $[7]$,  cryptography $[8]$  etc.

Let $O_1$ and $O_2$ be two sets where $O_1$ is the set of all finite semisimple group algebras and $O_2$ is the set of all rings of the type $\oplus_{t=1}^jM_{n_t}(\mathbb{F}_t)$ where $j$ is arbitrary integer and $\mathbb{F}_t$ is arbitrary finite field for each $t$(here we assume that all the finite fields $\mathbb{F}_t$ for each $t$ have same characteristic). Let $\psi: O_1\to O_2$ is the direct map which maps each semisimple group algebra to its unique Wedderburn decomposition. Many researchers have devoted a lot of effort to determine the Wedderburn decomposition of  semisimple group algebras, for instance $[2, 3, 4, 5,  6 ]$ etc. However, inverse of $\psi$, i.e. corresponding to a  ring of the type $\oplus_{t=1}^jM_{n_t}(\mathbb{F}_t)$, do there exist a semisimple group algebra which is mapped to $\oplus_{t=1}^jM_{n_t}(\mathbb{F}_t)$ by $\psi$, is less clear.   From $[5, \ \text{Lemma}\ 2.5]$, we know that $\mathbb{F}_q$ is always one of the  component in  Wedderburn decomposition of the semisimple group algebra $\mathbb{F}_qG$. Therefore, we conclude that the above mentioned inverse problem is ill-posed as for example, all the rings of the type $\oplus_{t=1}^jM_{n_t}(\mathbb{F}_t)$ where $n_t>1$ has no inverse image under $\psi$. Recall that a problem is ill-posed if it has either no solution or more than one solutions (there is also a third condition  related to continuity of $\psi$, but we are not concerned about it).

Given a ring of the type $\oplus_{t=1}^jM_{n_t}(\mathbb{F}_t)$, we say it to be complete (or Wedderburn decomposition),  if there exists a group algebra $\mathbb{F}_qG$  such that upon addition of some  simple components in $\oplus_{t=1}^jM_{n_t}(\mathbb{F}_t)$, it becomes image of $\mathbb{F}_qG$ under $\psi$ where $\mathbb{F}_qG$  is smallest such group algebra. For example, there is no group algebra $\mathbb{F}_qG$  such that $\psi(\mathbb{F}_qG)= M_2(\mathbb{F}_q)^3\oplus M_3(\mathbb{F}_q)$, where $q=p^n$ and $p\geq 5$ is a prime. However, if we add $\mathbb{F}_q^3$ to $M_2(\mathbb{F}_q)^3\oplus M_3(\mathbb{F}_q)$, it becomes complete as for $G=SL(2, 3)$, $\psi(\mathbb{F}_qG)= \mathbb{F}_q^3\oplus M_2(\mathbb{F}_q)^3\oplus M_3(\mathbb{F}_q)$ $[5]$. 
Next, we conjecture a result about the completeness of a ring which   is also the main result of the paper.
\begin{conj}
Let $p$ is  some prime and $j \in \mathbb{Z}$. Then, every ring of the type $\oplus_{t=1}^jM_{n_t}(\mathbb{F}_t)$, where $\mathbb{F}_t$ for each $t$ is a finite field of characteristics $p$,  is complete.
\end{conj}
If above conjecture is true, then our inverse problem becomes well-posed in the sense of existence of solution. However, there may exist more than one completeness of a single ring. Currently, we do not have any proof for above-said conjecture, but to prove it one need to look for a group algebra $\mathbb{F}_qG$ for some $q$ such that $\psi(\mathbb{F}_qG)= \oplus_{t=1}^sM_{n_t}(\mathbb{F}_t)$ where $ \oplus_{t=1}^sM_{n_t}(\mathbb{F}_t)$ is completeness of  $\oplus_{t=1}^jM_{n_t}(\mathbb{F}_t)$ . Next result can be very helpful in the determination of such group algebra.
\begin{prop}
Let $\mathbb{F}_qG$ is  a semisimple group algebra for some group $G$ where $q=p^k, k \geq 1$ and   $p$ is some prime. Further, let  $\oplus_{t=1}^jM_{n_t}(\mathbb{F}_t)$ is a ring with $j \in \mathbb{Z}$ and $n_t\geq 1$ for each $t$. Then $\oplus_{t=1}^jM_{n_t}(\mathbb{F}_t)$ is Wedderburn decomposition of $\mathbb{F}_qG$ provided all of the following hold:
\begin{enumerate}
\item for each $t$, $\mathbb{F}_t$ is a finite extension of $\mathbb{F}_q$.
\item if $G'$ is commutator subgroup of $G$, then  $\mathbb{F}_q(G/G')$ is isomorphic to sum of all commutative components of $\oplus_{t=1}^jM_{n_t}(\mathbb{F}_t)$.
\item this result tells us about $j$ in $\oplus_{t=1}^jM_{n_t}(\mathbb{F}_t)$.
Let $g \in G$ be a $p$-regular element and    the sum of all conjugates of $g$ be denoted by $\gamma_g$. Further, let the cyclotomic $\mathbb{F}_q$-class of $\gamma_g$ be denoted by $$S(\gamma_g) = \{ \gamma_{g^n} \ |\  n \in I_{\mathbb{F}}\}$$ where $$I_{\mathbb{F}} = \{ n \ |\ \zeta \mapsto \zeta^n \ \text{is an automorphism of }  \  \mathbb{F}_q(\zeta) \ \text{over}\ \mathbb{F}_q\},$$ $\zeta$ is primitive $e^{\text{th}}$ root of unity,  $e$ is exponent of $G$. Then, the number of components in  the ring $\oplus_{t=1}^jM_{n_t}(\mathbb{F}_t)$ is equal to number of cyclotomic $\mathbb{F}_q$-classes in $G$.
\item this result tells us about $\mathbb{F}_t$ in $\oplus_{t=1}^jM_{n_t}(\mathbb{F}_t)$. Let $\text{Gal}(\mathbb{F}_q(\zeta)/\mathbb{F}_q)$ be  cyclic and  $j$ be the number of cyclotomic $\mathbb{F}_q$-classes in $G$. If $K_i, 1 \leq i \leq j$, are the simple components of center of $\mathbb{F}_qG$ and $S_i, 1 \leq i \leq j$, are the cyclotomic $\mathbb{F}_q$-classes in $G$, then $|S_i| = [K_i:\mathbb{F}_q]$ for each $i$ after suitable ordering of the indices if required. 
\end{enumerate}
\end{prop}
\begin{proof}
Part $(1)$ is trivial. For part $(2)$, see $[1, \ \text{proposition}\ 3.6.11]$. For part $(3)$ and $(4)$, see $[9]$.
\end{proof}
\section{Discussion}
As already discussed, units of a  group algebra have several applications. If Conjecture $1$ is true, then we can take a ring of $\oplus_{t=1}^jM_{n_t}(\mathbb{F}_t)$ of our choice and employ the units of completed ring in various applications. Especially in the field of cryptography, availability of large number of variety of units can be help in proposing new cryptosystems or improve the security of exisiting cryptosystems on group rings $[8]$. Further, if Conjecture $1$ is true, then the next important task is to measure the hardness of completeness problem, i.e. how hard it is to find the completeness of a given ring.
\bibliographystyle{plain}

\end{document}